\documentclass[11pt]{amsart}

\usepackage{amsmath,amsthm,amsfonts,amssymb,mathrsfs}
\usepackage{inputenc,mathrsfs}

\numberwithin{equation}{section}
\usepackage[dvips]{graphicx}
\usepackage[colorlinks]{hyperref}

\newtheorem{definition}{Definition}[section]
\newtheorem{theorem}{Theorem}[section]
\newtheorem{proposition}{Proposition}[section]
\newtheorem{lemma}{Lemma}[section]
\newtheorem{corollary}{Corollary}[section]

\theoremstyle{remark}
\newtheorem{remark}{Remark}[section]

\newtheorem*{acknowledgements}{Acknowledgements}
\newtheorem*{MSC}{Mathematics subject clasification (2010)}

\DeclareMathOperator{\diam}{\textrm{diam}}
\DeclareMathOperator{\Rm}{\textrm{Rm}}
\DeclareMathOperator{\ric}{\textrm{Ric}}
\DeclareMathOperator{\hess}{\textrm{Hess}}
\DeclareMathOperator{\rf}{\mathcal{RF}^{C,n}}
\DeclareMathOperator{\rfr}{\mathcal{RF}^{C,n}_{reg}}
\DeclareMathOperator{\rv}{\mathcal V}
\DeclareMathOperator{\vol}{\textrm{Vol}}

\title[The size of the singular set of a Type I Ricci flow]{The size of the singular set of a Type I Ricci flow}

\author[P. Gianniotis]{Panagiotis Gianniotis}

\address{Department of Mathematics\\ 
University College London\\
25 Gordon St, London WC1E 6BT, United Kingdom}
\email{p.gianniotis@ucl.ac.uk}

\begin{document}
\begin{abstract}
In a singular Type I Ricci flow, we consider a stratification of the set where there is curvature blow-up,  according to the number of the Euclidean factors split by the tangent flows. We then show that the strata are characterized roughly in terms of the decay rate of their volume, which in our context plays the role of a dimension estimate.
 \begin{MSC}
 Primary 53C44; Secondary 35K55, 58J35
\end{MSC}
\end{abstract}

\maketitle

\section{Introduction}
A one parameter family $(g(t))_{t\in [-T,0)}$ of Riemannian metrics on a compact manifold $M^n$ is a Ricci flow if it satisfies the evolution equation
\begin{eqnarray}
 \frac{\partial}{\partial t} g(t)=-2\ric(g(t)). \label{rf_eqn}
\end{eqnarray}
  
If  $(g(t))_{t\in [-T,0)}$ can not be extended smoothly past time $t=0$ then the flow is singular and $\sup_{M\times [-T,0)} |\Rm(g(t))|_{g(t)} = \infty$.
We call a singular Ricci flow \textit{Type I} if there is a $C>0$ such that 
\begin{eqnarray}
 \sup_M |\Rm(g(t))|_{g(t)}\leq \frac{C}{|t|},
\end{eqnarray}
for all $t\in [-T,0)$. For Type I flows, Naber \cite{Naber} shows that the Cheeger-Gromov limit $(N, h(t),q)_{t\in (-\infty,0)}$ of any  blow-up sequence of the form $(M,\tau_i^{-1}g(\tau_it), p)$ , where $\tau_i\rightarrow 0$,  is a gradient shrinking Ricci soliton. Namely there exists an $f\in C^\infty (N)$ such that 
\begin{eqnarray}
 \ric(h(-1))+\hess_{h(-1)} f=\frac{h(-1)}{2},
\end{eqnarray}
We will call such a limit \textit{a tangent flow of $g(t)$ at $p$}. 

Enders, M\"uller and Topping \cite{EMT} then show that tangent flows at singular points of the Ricci flow are necessarily non-flat. Here the set of singular points is defined as follows.

\begin{definition}\label{singular_set}
 A point $p\in M$ is at the singular set $\Sigma$ of $(M,g(t))_{t\in [-T,0)}$ if there is no neighbourhood $U$ of $p$ such that
\begin{eqnarray}
 \sup_{U\times [-T,0)} |\Rm(g(t))|_{g(t)} < \infty.
\end{eqnarray}
\end{definition}

Very little is known regarding the structure of the singular set $\Sigma$ or its behaviour as the flow approaches the singular time. In the setting of the K\"ahler Ricci flow on a compact K\"ahler manifold $X$, Collins and Tosatti \cite{Collins} prove  a conjecture of Feldman, Ilmanen and Knopf \cite{FIK}, that if $\Sigma$ is a proper subset of $X$ then it is the union of irreducible analytic subvarieties of $X$ whose volume in the respective dimension decays to zero. Moreover, $\Sigma$ is a subset of real codimension at least two.

For a general Type I Ricci flow, such precise understanding of the singular set is not yet available. It is shown however in \cite{EMT} that the $n$-dimensional volume of the singular set should decay to zero as the flow approaches the singular time. Namely, if $\vol_{g(0)} \Sigma <\infty$ then $\lim_{t\rightarrow T}\vol_{g(t)} \Sigma=0$. 

This decay motivates our work in two different ways. First of all, it raises the question of the rate of this decay. Note that, in principle, estimates on the volume of high curvature or singular regions could lead to $L^p$ curvature estimates along a Type I Ricci flow. From a different point of view, the decay of the volume of the singular set may be seen as an estimate for the ``dimension" of $\Sigma$ at the singular time.

Dimension estimates for singular sets of geometric PDEs have a long history. For instance, we have dimension estimates for the singular set of mass minimizing integral currents (see Federer \cite{Federer} and Almgren \cite{Almgren}), as well as for energy minimizing maps (see Schoen-Uhlenbeck \cite{SU} and Simon \cite{SimonL}). Moreover, White in \cite{White} proves very general stratification theorems for upper-semicontinuous functions on domains in Euclidean space, which put the previous results in a general framework and make the theory applicable in a variety of contexts, including the mean curvature flow.  Last but not least, one has the dimension estimates for the singular set of non-collapsed limits of Riemannian manifolds with Ricci curvature bounded below, arising from the theory of Cheeger and Colding in \cite{CheegerColding}. 

In general such stratification theorems involve decomposing the singular set $\Sigma$ into an ascending sequence $\Sigma_0 \subseteq\Sigma_1 \subseteq ... \subseteq \Sigma_N=\Sigma$ (for some $N\geq 0$) and then proving Hausdorff dimension estimates of the form $\dim \Sigma_j\leq j$. In this article, inspired particularly from \cite{White}, we prove a stratification theorem for the singular set of a Type I Ricci flow. 

Let $\Sigma$ be the singular set of a Type I Ricci flow $(M,g(t))_{t\in [-T,0)}$, as in Definition \ref{singular_set}, and for every $j=0,\ldots,n-2$ set
\begin{eqnarray}
 \Sigma_j&=&\{ x\in \Sigma, \textrm{ no tangent flow at } x \textrm{ splits as}\; (N^{n-j-1},h(t))\times (\mathbb R^{j+1},g_{Eucl})  \}. \nonumber
\end{eqnarray}
It is clear that $\Sigma_0\subseteq  \Sigma_1\subseteq \cdots\subseteq \Sigma_{n-2}\subseteq\Sigma$. 

The main result provides analogues for Type I Ricci flows of the Hausdorff dimension estimates mentioned above. However, several subtleties arise when we attempt to make sense of such estimates in the case of the Ricci flow, since the aid of an ambient space is no longer available. Namely, although Ricci flow shrinks a round $n$-sphere to a point and we would like to regard such singularity as $0$-dimensional, the singular set $\Sigma$ is the whole manifold.
 
One way to remedy this issue would be to study instead the limiting space of the flow as it approaches the singular time. However, it is not known whether a singular Ricci flow $(M,g(t))_{t\in[0,T)}$ (even a Type I flow) converges to a metric space $(X,d_X)$ as $t$ approaches the singular time $T$ in general. Alternatively, one could make sense of the concept of the singular set and its dimension by embedding the flow in a larger ambient space, see \cite{KL1} and \cite{HN}. However, this is beyond the scope of the present paper.  

In the following, we interpret the dimensionality of a singular stratum $\Sigma_j$ via volume decay estimates, observing that along a cylindrical Ricci flow $g(t)=-2(n-j-1)t g_{S^{n-j}} + g_{Eucl}$ on $S^{n-j}\times \mathbb R^j$ the volume form is given by $d\mu_{g(-\tau)}=\tau^{\frac{n-j}{2}} d\mu_{g(-1)}$.

\begin{theorem}\label{main_theorem}
Fix $j=0,\ldots, n-2$ and let $\varepsilon >0$. Then, there exist closed $A_i\subset \bar{\Sigma}_j$ ($i=1,2,\ldots$), depending on $j$ and $\varepsilon$, such that $\Sigma_j \subset \bigcup_{i=1}^\infty A_i$ and  
\begin{eqnarray}
 \frac{\vol_{g(-\tau)} (A_i) }{\tau^{\frac{n-j-\varepsilon}{2}}} &\leq& C(j,\varepsilon,i)\tau^{\beta},\label{size_est_1}
\end{eqnarray}
for some $\beta=\beta(\varepsilon)\in (0,1)$. Also, for every $\delta>0$ there is $i_0$ such that 
\begin{eqnarray}
 \vol_{g(-\tau)}(\bar\Sigma_j\setminus A_i) <\delta, \label{small_vol}
\end{eqnarray}
for every $i\geq i_0$ and $\tau\in (0,T]$.

Moreover, for each $x \in \Sigma_0$ there exist $R_0,\bar\tau >0$  such that
\begin{eqnarray}
B_g(x,-\bar\tau, R_0\sqrt{\bar\tau}) \cap \{y\in M,\;\Theta_g(y)\leq \Theta_g(x)\} \subseteq B_g(x,- \tau, R_0\sqrt\tau), \label{size_est_3}
\end{eqnarray}
for every $\tau \in (0,\bar\tau]$. 
\end{theorem}

Here, $B_g(x,t,r)$ denotes the $g(t)$-metric ball of radius $r$ centered at $x\in M$, $\vol_{g(t)}$ is the $n$-dimensional Hausdorff measure with respect to $g(t)$ and $\Theta_g(\;\cdot\; )$ is a lower semicontinuous function on $M$, analogous to the Gaussian density for the mean curvature flow, which is defined in Section \ref{def_of_density}. 

Observe that estimate (\ref{size_est_3}) may be interpreted as the isolatedness of points in the $0$-dimensional stratum $\Sigma_0$ with a fixed density value. Of course such statement taken literally can not be true. For instance, for the shrinking round sphere $S^n$, $\Sigma_0=S^n$ and all points have the same density due to symmetry. On the other hand, the diameter goes to zero and all points may be thought to represent a single singular point.

Idealy, in Theorem \ref{main_theorem} we would prefer an estimate on the volume of $\Sigma_j$ instead of the sets $A_i$.  These sets arise by the decomposition of $\Sigma_j$ according to the scale below which the flow is sufficiently close to a shrinking soliton. It is below that scale that our argument allows to iteratively refine a given covering, making it more efficient as the flow approaches the singular time. On the other hand, an estimate on the volume on the whole $\Sigma_j$ would be more in the spirit of Minkowski content estimates. Such estimates for singular sets have recently been obtained using quantitative differentiation  arguments in different contexts (see for instance \cite{ChNab1}, \cite{ChNab2}, \cite{ChHasNab1}, \cite{ChHasNab2}, \cite{ChNabVAl}). We intend to explore this direction further in a future paper.

Finally, the following corollary partially improves the volume decay statement in \cite{EMT} exploiting the fact that a shrinking Ricci soliton splitting more than $n-2$ Euclidean factors should necessarily be flat. Moreover, when the Weyl tensor remains bounded along the flow we obtain an improved volume decay estimate, as a consequence of the fact that Weyl-flat gradient shrinking Ricci solitons can split at most one Euclidean factor, which follows from  \cite{Weylflat}.  

\begin{corollary}\label{corol}
If $\Sigma=\Sigma_j$ it follows that for every $\varepsilon>0$ there exist closed $A_i\subseteq \Sigma$, $i=1,2,...$, such that $\Sigma=\bigcup_{i=1}^\infty A_i$ and
\begin{eqnarray}
\vol_{g(-\tau)}(A_i) \leq C(i,\varepsilon) \tau^{\frac{n-j}{2}-\varepsilon}, \label{cor_est}
\end{eqnarray}
for every $\tau\in (0,T]$. In particular we distinguish the following cases. 
\begin{enumerate}
\item In general $\Sigma= \Sigma_{n-2}$ and $\vol_{g(-\tau)}(A_i) \leq C(i,\varepsilon) \tau^{1-\varepsilon}$.
\item Suppose that the Weyl curvature satisfies 
\begin{eqnarray}
\sup_{M\times [-T,0)} |W_g|_g < \infty. \nonumber
\end{eqnarray}
Then $\Sigma=\Sigma_1$ and $\vol_{g(-\tau)}(A_i) \leq C(i,\varepsilon) \tau^{\frac{n-1}{2}-\varepsilon}$.
\end{enumerate}
In both cases, for every $\delta>0$ there is $i_0$ such that $\vol_{g(-\tau)}(\Sigma\setminus A_i) <\delta$ for every $i\geq i_0$ and $\tau\in (0,T]$.
\end{corollary}

The outline of the paper is as follows. In Section \ref{preliminaries} we collect a few preliminary facts. Then, in Section \ref{def_of_density} we introduce a monotone quantity which plays the role of Perelman's reduced volume based at the singular time, and its associated density function. They are both lower-semicontinuous under the Cheeger-Gromov convergence of Ricci flows, which is essential to our arguments. In Section \ref{spl}, given any non-flat gradient shrinking Ricci soliton, we distinguish the set of points where the density function above achieves its minimum, called the spine. We then prove a splitting theorem (Theorem \ref{splitting}), which asserts that the soliton splits enough Euclidean factors $\mathbb R^j$, $0\leq j\leq n-2$, so that its spine is of the form $V\times \mathbb R^j$ and the diameter of $V$ decays to zero as the flow induced by the soliton approaches the singular time. Finally, in Section \ref{sizeof} we prove Theorem \ref{main_theorem} via a covering argument similar to \cite{SimonL}.

\begin{acknowledgements}
The author would like to thank Alix Deruelle, Felix Schulze and Peter Topping for many interesting discussions and their valuable support. This research has been supported by the grant of the German Science Foundation entitled  ``Regularity and stability of curvature flows and their applications to geometric variational problems". The author would also like to acknowledge support by  the EPSRC, on a programme grant entitled ``Singularities of Geometric Partial Differential Equations'' reference number EP/K00865X/1, during the first stages of the project.
\end{acknowledgements}

\section{Preliminaries}\label{preliminaries}
In this section we collect some preliminary results on which we will rely in the rest of the paper.

\subsection{Gradient shrinking Ricci solitons}
A triplet $(M^n,g,f)$ is called a gradient shrinking Ricci soliton if it satisfies the equation
\begin{eqnarray}
\ric_g+\hess_g f = \frac{g}{2}.\label{soliton}
\end{eqnarray}

Clearly, if $(M,g)$ is complete with bounded curvature, the vector field $\nabla^g f$ is complete. It then follows from (\ref{soliton}) that the shrinking Ricci soliton $(M,g,f)$  induces a Ricci flow $h(t)=-t \phi_t^* g$ on $M$, where $\phi_t$ are the diffeomorphisms generated by $\nabla^g f$ via
\begin{eqnarray}
\frac{d}{dt}\phi_t&=&-\frac{1}{t} \nabla^g f \circ \phi_t , \nonumber\\
\phi_{-1}&=&id_M.\nonumber
\end{eqnarray}

It is well known that the following identity holds
\begin{eqnarray}
R+|\nabla f|^2-f&=& c. \label{soliton_identity}
\end{eqnarray}
The soliton function $f$ is well-defined up to a linear function. However, when $c=0$ in (\ref{soliton_identity}),  we call $f$  a \textit{normalized soliton function}. Normalized soliton functions will be important to us mainly because of the following result from \cite{Naber}.
\begin{lemma}[Lemma 2.1 in \cite{Naber}]\label{f_vol}
Let $(M,g,f), (M',g',f')$ be normalized shrinking solitons and suppose that
\begin{enumerate}
\item $\int_M e^{-f} d\mu_g, \int_{M'} e^{-f'} d\mu_{g'} < +\infty$,
\item $(M,g)$ and $(M',g')$ are isometric.
\end{enumerate}
Then, $\int_M e^{-f} d\mu_g= \int_{M'} e^{-f'} d\mu_{g'}$.
\end{lemma}

\subsection{The reduced distance and volume under Type I curvature bounds}
\begin{definition}[Type I Ricci flow]
For every positive integer $n$ and $C>0$ we define the following classes of pointed complete Ricci flows
\begin{eqnarray}
 \rf&=&\{ (M^n,g(t),p)_{t\in (-T,0)},\; g(t) \textrm{\;solves\;} (\ref{rf_eqn})\;\textrm{and}\;|\Rm_g|_g\leq \frac{C}{|t|}\; \textrm{on}\; M\times (-T,0) \}. \nonumber\\
 \rfr&=&\{(M^n,g(t),p)\in \rf,\; \sup_{ M\times (-T,0)}|\Rm_g|_g<+\infty \}. \nonumber
\end{eqnarray}
 \end{definition}
 
 We equip  $\rf$ with the topology of smooth ($C^\infty$) pointed Cheeger-Gromov-Hamilton convergence for Ricci flows (uniformly in compact sets). Since for any $\mathfrak g=(M,g(t),p)_{t\in (-T,0)}$ and $T_i>0$ such that $T_i\searrow 0$,  the sequence $\mathfrak g_i=(M,g(T_i +t),p)_{t\in(-T+T_i,0)}$ converges to $\mathfrak g$, it follows that $\overline{\rfr}=\mathcal{RF}^C$.
 
 \begin{definition}\label{reduced_distance}
 For $\mathfrak g=(M,g(t),p)_{t\in [-T,0]}\in \rfr$, the reduced distance function, originaly defined in \cite{Perelman1}, is given by
 \begin{eqnarray}
 l_{\mathfrak g}(x,\bar\tau)=\inf_\gamma\left\{ \frac{1}{2\sqrt{\bar\tau}} \int_0^{\bar\tau}\sqrt \tau (R_{g(-\tau)} (\gamma(\tau) )+\left|\frac{d\gamma}{d\tau}\right|_{g(-\tau)}^2  ) d\tau   \right\}, \label{r_dist} 
 \end{eqnarray}
 where $\gamma: [0,\bar\tau]\rightarrow M$, $\gamma(0)=p$, $\gamma(\bar\tau)=x$.
 \end{definition}
 The following estimate from \cite{Naber} will be crucial to  our work.
 
\begin{proposition}\label{red_dist_est}
 Let $\mathfrak g=(M,g(t),p)\in \rfr$. There exists $A=A(n,C)>0$ such that
\begin{enumerate}
 \item $\frac{1}{A}(1+\frac{d_{g(-\tau)}(p,x)}{\sqrt{\tau}})^2 - A \leq l_{\mathfrak g}(x,\tau) \leq A (1+\frac{d_{g(-\tau)}(p,x)}{\sqrt\tau})^2$,
 \item $|\nabla l_{\mathfrak g}|(x,\tau)\leq \frac{A}{\sqrt{\tau}} (1+\frac{d_{g(-\tau)}(p,x)}{\sqrt{\tau}})$,
 \item $|\frac{\partial l_{\mathfrak g}}{\partial \tau}|(x,\tau)\leq \frac{A}{\tau} (1+\frac{d_{g(-\tau)}(p,x)}{\sqrt{\tau}})^2$.
\end{enumerate}
\end{proposition}

 Now, whenever $\mathfrak g_i\rightarrow \mathfrak g$, with $\mathfrak g_i,\mathfrak g\in \rf$, it is possible to limit out the reduced distance functions of $\mathfrak g_i$ using the estimates in Proposition \ref{red_dist_est}, as is done in \cite{Naber}. This motivates the following definition.

\begin{definition}
Let $\mathfrak g=(M, g(t),p)_{t\in (-T,0)}\in \rf$. A function $l: M\times (0, T)\rightarrow \mathbb R$ is called a singular reduced distance if the following holds. There exists a sequence $\mathfrak g_i\in \rfr$ converging to $\mathfrak g$ in the topology of $\rf$ such that $l_{g_i}\rightarrow l$ in $C^{0,\alpha}_{loc}$.
\end{definition}

\begin{remark}\label{rmk_s_r_d}
Since $\overline{\rfr}=\mathcal{RF}^C$ it follows that the set of singular reduced distance functions of $\mathfrak g\in\rf$ is non-empty. Moreover, since the estimates of Proposition \ref{red_dist_est} pass to the limit, it follows that the space of singular reduced distance functions of $\mathfrak g\in\rf$ is compact in the $C^{0,\alpha}_{loc}$ topology.
\end{remark}

Given a singular reduced distance $l$ on $\mathfrak g$ and $\tau\in (0,T)$, following \cite{Perelman1} we define the reduced volume associated to $l$ as
\begin{eqnarray}
\rv_{\mathfrak g,l}(\tau):=\int_M (4\pi \tau)^{-\frac{n}{2}} e^{-l(\cdot,\tau)} d\mu_g.
\end{eqnarray}
\begin{remark}
Note that if $\mathfrak g\in \rfr$ then any singular reduced distance function $l$ of $\mathfrak g$ is given by (\ref{r_dist}). To see this, consider $\mathfrak g_i\rightarrow \mathfrak g$, with $\mathfrak g_i,\mathfrak g\in \rfr$. By Perelman's pseudolocality theorem (see \cite{Perelman1}) it follows that $\mathfrak g_i$ have uniformly bounded curvature in time intervals $[-a,0]$, $a<T$,  hence they converge to $\mathfrak g$ uniformly locally in  $M\times (-T,0]$. Hence, the reduced distance functions $l_{\mathfrak g_i}$ pointwize converge to $l_{\mathfrak g}$.
\end{remark}
In the following lemma we collect a few useful facts about $\rv_{\mathfrak g,l}(\; \cdot\;)$. 

\begin{lemma}\label{red_vol_props}
The reduced volume $\rv_{\mathfrak g,l}(\tau)$ with respect to a singular reduced distance $l$ of $\mathfrak g\in \rf$ is non-increasing in $\tau$. Moreover, if there exist $0<\tau_1<\tau_2$ such that
\begin{eqnarray}
\rv_{\mathfrak g, l}(\tau_1)=\rv_{\mathfrak g, l}(\tau_2),
\end{eqnarray}  
then $g(t)$ is a gradient shrinking Ricci soliton, namely 
\begin{eqnarray}
\ric(g(-\tau))+\hess_{g(-\tau)} l(\cdot,\tau)=\frac{1}{2\tau} g(-\tau),\label{soliton_eqn}
\end{eqnarray}
and $l(\cdot,1)$ is a normalized soliton function.
\end{lemma}
\begin{proof}
The monotonicity statement is essentially Lemma 2.8 in \cite{Naber}. The statement (\ref{soliton_eqn}) and the fact that $l(\cdot,1)$ is a normalized soliton function are proven together with Theorem 2.1 in \cite{Naber}.
\end{proof}

\subsection{The non-inflating property of the Ricci flow.}
Now we recall the non-inflating propety of smooth compact Ricci flows, as it appears in Zhang \cite{Zhang}. A similar result was also obtained by Chen and Wang in \cite{WangChen} under additional assumptions. The result in \cite{Zhang} however is more suitable for the setting of Type I Ricci flows.

\begin{theorem}[Theorem 1.1 in \cite{Zhang}]\label{noninflatingthm}
Let $(M^n,g(t))_{t\in [0,t_0]}$ be a smooth and compact Ricci flow. Then for every $\alpha>0$ there exists a $\kappa>0$, depending on $n,\alpha, g(0)$, with the following property. If for some $x_0\in M$ and $r\in (0,\sqrt{t_0})$ the estimate
\begin{eqnarray}
 R(g(t))\leq \frac{\alpha}{t_0-t},
\end{eqnarray}
holds in $B_g(x_0,t_0,r)\times [t_0-r^2,t_0]$, then 
\begin{eqnarray}
 \vol_{g(t_0)}(B_g(x_0,t_0,r))\leq \kappa r^n.
\end{eqnarray} 
\end{theorem}

The non-inflating property has the following consequence for Type I flows, which is a direct consequence of Theorem \ref{noninflatingthm} and the Type I curvature bound.

\begin{corollary}
 Let $(M,g(t),x_0)_{t\in [-T,0)}\in \rf$. Then, there exists a $\kappa_0>0$ depending on $n,g(-T),C$, such that for every $r\in (0,\sqrt{\frac{T}{2}})$ and $t_0\in (-\frac{T}{2},0)$
\begin{eqnarray}
 \vol_{g(t_0)}(B_g(x_0,t_0,r))\leq \kappa_0 r^n.\label{noninflate}
\end{eqnarray}
\end{corollary}
\begin{proof}
 Since $(M,g(t),x)_{t\in [-T,0)}\in \rf$, there exists $c(n,C)>0$ such that $R(g(t))\leq \frac{c(n,C)}{|t|} \leq \frac{c(n,C)}{t_0-t}$ on $M\times [t_0-r^2,t_0]$. Estimate (\ref{noninflate}) then follows from Theorem \ref{noninflatingthm}. 
\end{proof}

\section{The reduced volume based at the singular time}\label{def_of_density}

In this section we define a monotone quantity which plays the role of a reduced volume based at a singular time and consider the associated density function. 

In particular, the compactness property of Remark \ref{rmk_s_r_d}  allows the following definition.
\begin{definition}\label{singular_red_vol}
We define the singular reduced volume at scale $\tau>0$  of $\mathfrak g\in \rf$ as
\begin{eqnarray}
\rv_{\mathfrak g}(\tau)=\rv_{\mathfrak g, \bar l}(\tau),\nonumber
\end{eqnarray}
where $\bar l$ is the minimizer of $\rv_{\mathfrak g,l}(\tau)$ among all singular reduced distances $l$ of $\mathfrak g$.
\end{definition}

A direct implication of this definition is the following.
\begin{proposition}\label{rv_props}
Let $ \mathfrak g=(M,g(t),p)_{t\in (-T,0)} \in \rf$. Then
\begin{enumerate}
\item If  $0<\tau_1<\tau_2$ then $\rv_{\mathfrak g}(\tau_1) \geq \rv_{\mathfrak g}(\tau_2)$.
\item If $\rv_{\mathfrak g}(\tau_1) = \rv_{\mathfrak g}(\tau_2)$ for some $0<\tau_1<\tau_2$, then there exists a singular reduced distance $l$ of $\mathfrak g$ such that $\rv_{\mathfrak g}(\tau)=\rv_{\mathfrak g,l}(\tau)$, for every $\tau\in (0,T)$. Moreover, $l$ is a normalized soliton function.
\item Let $\mathfrak g_i=(M_i,g_i(t),p_i)\in \rf$ such that $\mathfrak g_i\rightarrow \mathfrak g$. Then for every $\tau\in(0,T)$
\begin{eqnarray}
\liminf_i \rv_{\mathfrak g_i}(\tau)\geq \rv_{\mathfrak g} (\tau). \nonumber
\end{eqnarray}
\end{enumerate}
\end{proposition}

In particular, the monotonicity property in Proposition \ref{rv_props} justifies the following definition.
\begin{definition}
The density of $\mathfrak g\in \rf$ is defined as
\begin{eqnarray}
\Theta_{\mathfrak g}:=\lim_{\tau \searrow 0} \rv_{\mathfrak g} (\tau).
\end{eqnarray}
\end{definition}

\begin{remark}
 By Proposition \ref{rv_props},  if $\mathfrak g_i,\mathfrak g\in \rf$ and $\mathfrak g_i\rightarrow \mathfrak g$ it follows that
\begin{eqnarray}
\liminf_{i\rightarrow \infty} \Theta_{\mathfrak g_i}\geq \Theta_{\mathfrak g}.
\end{eqnarray}
\end{remark}
\begin{remark}\label{density_values}
 For every $\mathfrak g\in \rfr$ the reduced volume satisfies $\mathcal V_{\mathfrak g,l_{\mathfrak g}}(\tau) \in (0,1]$ for every $\tau>0$, since $\lim_{\tau\rightarrow 0^+} \mathcal V_{\mathfrak g,l_{\mathfrak g}}(\tau) =1$ (see for instance \cite{ricciflow}). Thus, $\Theta_{\mathfrak g}\in [0,1]$ for every $\mathfrak g\in \rf$. In fact, $\Theta_g >0$, due to Perelman's no-local collapsing theorem.
\end{remark}

\begin{definition}
Let $(M,g(t))_{t\in (-T,0)}$ be a Type I Ricci flow and   $x\in M$. We will always denote $\mathfrak g_x:=(M,g(t),x)_{t\in (-T,0)}$.  Suppose that for some $C>0$, $\mathfrak g_x \in \rf$. Then, the density of $g$ at $x$ is defined as 
\begin{eqnarray}
\Theta_g(x)=\Theta_{\mathfrak g_x}.
\end{eqnarray}
\end{definition}

\begin{remark}
 In \cite{EMT} the authors introduce a different notion of reduced volume based at the singular time and density function for a Type I Ricci flow. In their approach, they consider limits of reduced distance functions arising from different sequences of times converging to the singular time. Then, they define a singular reduced distance by considering the infimum over all these limits, and use this to build a monotone quantity and its associated density. It is not clear to the author if the density function in \cite{EMT} behaves well under Cheeger-Gromov convergence. Instead, the lower semicontinuity is essentially built in the definition of our density.
\end{remark}

The reduced volume  based at the singular time involves minimizing over all approximating sequences of Ricci flows, and in principle is hard to compute. However, for shrinking Ricci solitons, the following lemma provides some information.

\begin{lemma}\label{density_soliton}
Let $\mathfrak g:=(M,g(t),p)_{t\in (-\infty,0)}\in \rf$ be the Ricci flow associated to a normalized gradient shrinking Ricci soliton $(M,g(-1),f)$. Then, 
\begin{enumerate}
\item Let $l$ be an arbitrary singular reduced distance function for $\mathfrak g$. Then 
$$\lim_{\tau\rightarrow\infty} \rv_{\mathfrak g}(\tau)=\lim_{\tau\rightarrow\infty} \rv_{\mathfrak g,l}(\tau)=\int_M (4\pi)^{-\frac{n}{2}} e^{-f}d\mu_{g(-1)}.$$ 
\item If $p\in M$ is a critical point of $f$, then 
\begin{eqnarray}
\Theta_g(p)=\int_M (4\pi)^{-\frac{n}{2}} e^{-f}d\mu_{g(-1)} \leq \Theta_g(x), \nonumber
\end{eqnarray}
for every $x\in M$. In particular the density function $\Theta_g$ on a shrinking Ricci soliton attains a minimum.
\item If a singular reduced distance function $l$ for $\mathfrak g$ is also a soliton function, then 
\begin{eqnarray}
\rv_{\mathfrak g}(\tau)=\rv_{\mathfrak g,l}(\tau), \nonumber
\end{eqnarray}
for every $\tau>0$.
\end{enumerate}
\end{lemma}
\begin{proof}
The first statement is essentially Theorem 3.2 in \cite{Naber}. We describe its proof again here for completeness. 

Fix a $\bar\tau>0$, and let $l$ be a singular reduced distance for $\mathfrak g$. Take a sequence $\tau_i\rightarrow +\infty$ and define the blow-down sequence $\mathfrak g_i:=(M,\tau_i^{-1}g(\tau_i t), p)_{t\in(-\infty,0)}$ and  set $l_i(\cdot,\tau)=l(\cdot,\tau_i \tau)$. Note that $l_i$ is a singular reduced distance for $\mathfrak g_i$.

By monotonicity and the scaling behaviour of $l$ it follows that 
\begin{eqnarray}
\lim_{\tau\rightarrow \infty}\rv_{\mathfrak g,l}(\tau)=\lim_{i\rightarrow \infty} \mathcal V_{\mathfrak g,l}(\tau_i \bar\tau)=\lim_{i\rightarrow \infty} \mathcal V_{\mathfrak g_i, l_i}(\bar\tau). \label{limit}
\end{eqnarray}
On the other hand, there exists $q\in M$ such that $\mathfrak g_i \rightarrow \mathfrak g_q=(M,g(t),q)_{t\in(-\infty,0)}$. To prove this, observe that since $g(t)=-t\phi_t^* g(-1)$  
\begin{eqnarray}
\tau_i^{-1}g(\tau_i t)&=& (\phi_t^{-1}\circ\phi_{\tau_i t})^* g(t),\label{scaling1} \\
&=&\phi_{-\tau_i}^* g(t), \label{scaling2}   
\end{eqnarray}
since $\phi_{\tau_i t}= \phi_t\circ\phi_{-\tau_i}$.

Hence, the sequence $(M,\tau_i^{-1}g(\tau_i t), p)$ is isometric to $(M,g(t), \phi_{-\tau_i}(p))$. Now, since $\phi_{-\tau_i}(p)\rightarrow_i q$, where $q$ is a critical point of $f$, it follows that $\mathfrak g_i\rightarrow \mathfrak g_q:=(M,g(t),q)$.

Moreover,  there is a singular reduced distance $\bar l$ for $\mathfrak g_q$ such that $l_i \rightarrow \bar l$. Hence, using the estimates in Lemma \ref{red_dist_est} we conclude that for every $\bar\tau>0$
\begin{eqnarray}
\lim_{i\rightarrow \infty} \rv_{\mathfrak g_i,l_i}(\tau) =\mathcal V_{\mathfrak g_q,\bar l}(\bar\tau), \label{bar_l_normalized}
\end{eqnarray}
which together with (\ref{limit}) implies that $\bar l$ is a normalized soliton function, by Lemma \ref{red_vol_props}. Therefore,  using Lemma \ref{f_vol} we obtain
\begin{eqnarray}
\mathcal V_{\mathfrak g_q,\bar l}(\bar\tau)= \int_M (4\pi)^{-\frac{n}{2}} e^{-f}d\mu_{g(-1)}.\label{equal_f_vols}
\end{eqnarray} 
 Finally, combining (\ref{limit}), (\ref{bar_l_normalized}) and (\ref{equal_f_vols}) we obtain that
 \begin{eqnarray}
 \lim_{\tau\rightarrow \infty} \rv_{\mathfrak g,l}(\tau) = \int_M (4\pi)^{-\frac{n}{2}} e^{-f}d\mu_{g(-1)},\label{arv}
 \end{eqnarray}
  for every singular reduced distance function $l$ for $\mathfrak g$. This suffices to prove the first statement of the lemma.

 To prove the second assertion, let $p,x\in M$, $\nabla f(p)=0$ and denote, as usual, $\mathfrak g_p=(M,g(t),p)_{t\in(-\infty,0)}$, $\mathfrak g_x=(M,g(t),x)_{t\in(-\infty,0)}$. We will first prove that 
\begin{eqnarray}
 \Theta_g(p)=\lim_{\tau\rightarrow\infty} \mathcal V_{\mathfrak g_p}(\tau)=\int_M (4\pi)^{-\frac{n}{2}} e^{-f}d\mu_{g(-1)}. \label{density_ctl_point}
\end{eqnarray}

Consider $\tau_i\rightarrow 0$ and observe using (\ref{scaling1})-(\ref{scaling2}) that $\mathfrak g_i=(M,\tau_i^{-1}g(\tau_i t), p)$ is isometric to $(M,g(t), \phi_{-\tau_i}p)$. Since $p$ is a critical point of $f$, it follows that $\phi_{-\tau_i}(p)=p$, hence $\mathfrak g_i\rightarrow \mathfrak g_p$.

As in the proof of the first assertion of the lemma, let $l$ be an arbitrary singular reduced distance for $\mathfrak g$ and set $l_i(\cdot,\tau)=l(\cdot,\tau_i \tau)$. 

By the monotonicity of $\rv_{\mathfrak g,l}(\cdot)$ and the scaling behaviour of $l$, it follows that 
\begin{eqnarray}
 \lim_{\tau\rightarrow 0} \rv_{\mathfrak g,l}(\tau)=\lim_{i\rightarrow \infty} \rv_{\mathfrak g,l}(\tau_i \tau) =\lim_{i\rightarrow \infty} \rv_{\mathfrak g_i,l_i}( \tau).
\end{eqnarray}
Moreover, $\rv_{\mathfrak g_i,l_i}(\tau)\rightarrow \rv_{\mathfrak g,\bar l}(\tau)$, for some singular reduced distance $\bar l$ for $\mathfrak g$, since $\mathfrak g_i\rightarrow \mathfrak g$.

Therefore, $\rv_{\mathfrak g,\bar l}(\tau)=\lim_{\tau\rightarrow 0} \rv_{\mathfrak g,l}(\tau)$, for every $\tau>0$, which implies that $\bar l$ is a normalized soliton function. By Lemma \ref{f_vol} it follows that $\lim_{\tau\rightarrow 0}\rv_{\mathfrak g,l}(\tau)= \int_M (4\pi)^{-\frac{n}{2}}e^{-f}d\mu_g$. Since $l$ was arbitrary, this implies (\ref{density_ctl_point}).

The assertion of the lemma then follows from 
$$\Theta_g(x)\geq \lim_{\tau\rightarrow \infty} \rv_{\mathfrak g_x}(\tau)=\lim_{\tau\rightarrow \infty} \rv_{\mathfrak g_p}(\tau)= \Theta_g (p),$$
where again we used (\ref{arv}). Note that by estimates on the growth of the soliton function of a gradient shrinking Ricci soliton (see for instance \cite{Naber} or \cite{HM}), $f$ always has a critical point, hence $\Theta_g$ always attains a minimum.

 For the last assertion, note that if $l$ is a singular reduced distance for $\mathfrak g$ which is also a soliton function,  then by monotonicity  we obtain, for every $\tau>0$, 
\begin{eqnarray}
\rv_{\mathfrak g, l}(\tau)\geq \rv_{\mathfrak g}(\tau)\geq \lim_{\tau\nearrow \infty} \rv_{\mathfrak g}(\tau)=\lim_{\tau\nearrow \infty} \rv_{\mathfrak g, l}(\tau)= \rv_{\mathfrak g, l}(\tau).
\end{eqnarray}
Hence $\rv_{\mathfrak g}(\tau)=\rv_{\mathfrak g,l}(\tau)=\lim_{\tau\nearrow \infty} \rv_{\mathfrak g}(\tau)$.

\end{proof}

Given a shrinking Ricci soliton $(M,g(t))_{t\in (-\infty,0)}$, Lemma \ref{density_soliton} shows that the limit
\begin{eqnarray}
\lim_{\tau \rightarrow\infty} \rv_{\mathfrak g_x}(\tau) \label{rvlimit}
\end{eqnarray}
does not depend on  $x\in M$. This naturally leads to the following definition.

\begin{definition}
The asymptotic reduced volume from the singular time of a shrinking Ricci soliton $(M,g(t))_{t\in (-\infty,0)}$ is defined as
\begin{eqnarray}
\mathcal{ARV}(M,g):=\lim_{\tau \rightarrow\infty} \rv_{\mathfrak g}(\tau).
\end{eqnarray}
\end{definition}

\begin{remark}
The asymptotic reduced volume in the setting of ancient smooth (super)solutions to the Ricci flow was first studied by Yokota in \cite{Yokota}. However, the arguments in \cite{Yokota} do not seem to carry over to the setting of singular Type I flows, to show the invariance of the limit (\ref{rvlimit}) on the basepoint. For our work though, it suffices to establish this for shrinking Ricci solitons. 
\end{remark} 

\section{Splitting Ricci shrinkers.}\label{spl}

\begin{definition}
Let $(M, g(-1),f)$ be a gradient shrinking Ricci soliton, $g(t)$ the associated Ricci flow, and $p\in M$ a minimizer of  $\Theta_g$. The subset
\begin{equation}
 S(M,g)=\{x\in M,\;\; \Theta_g(x) =  \Theta_g(p)\},
\end{equation} 
will be called the spine of the gradient shrinking Ricci solition.
\end{definition}

\begin{remark}\label{closed}
The lower semi-continuity of the density implies that $S(M,g)$ is closed.
\end{remark}

\begin{lemma}[Splitting principle]\label{cone_splitting}
 Let $g(t)=g_M(t)+g_{Eucl}$, $t\in (-\infty,0)$,  be a gradient shrinking Ricci soliton on $M^k\times \mathbb R^{n-k}$, $0< k\leq n$, satisfying $|\Rm(g(-1))|_{g(-1)}\leq C$. Moreover, let $V\subseteq M$ such that $S(M\times \mathbb R^{n-k},g)=V\times \mathbb R^{n-k}$. Suppose there exist $\tau>0$ such that
 \begin{eqnarray}
 \frac{\diam_{g_M(-\tau)} (V)}{\sqrt \tau} > A\sqrt 2 -1, \label{assumption1} 
 \end{eqnarray}
where $A>0$ is given by  Proposition \ref{red_dist_est}. Then, there exists a gradient shrinking Ricci soliton $(N^{k-1},h(t))_{t\in (-\infty,0)}$ and $V'\subseteq N$ such that $(M,g_M(t))$ splits isometrically as $(N,h(t))\times (\mathbb R, g_{Eucl})$ and $S(M\times \mathbb R^{n-k},g)=V'\times \mathbb R^{n-k+1}$.
\end{lemma}
\begin{proof}
Since $S(M\times \mathbb R^{n-k},g)$ is closed, assumption (\ref{assumption1}) implies that there exist $x,y\in V$ satisfying
\begin{eqnarray}
\frac{d_{g_M(-\tau)}(x,y)}{\sqrt \tau} > A \sqrt 2 - 1. \label{assumption2}
\end{eqnarray}
Let $p=(x,0), q=(y,0) \in S(M\times \mathbb R^{n-k})$. By Lemma \ref{density_soliton}, $\Theta_g (p) = \Theta_g(q) = \mathcal{ARV}(M\times \mathbb R^{n-k})$. This implies that there exist singular reduced distance functions $l_p,l_q$ of $\mathfrak g_p$ and $\mathfrak g_q$ respectively, which are both soliton functions. 

It follows from the soliton equation that the difference $L(\cdot,\tau)=l_p(\cdot,\tau)-l_q(\cdot,\tau)$ satisfies $ \hess_{g(-\tau)} L(\cdot,\tau)=0$. Moreover, since the metric on $M\times \mathbb R^{n-k}$ splits,  its restriction $\bar L=L|_{M\times 0}$ satisfies $\hess_{g_M(-\tau)} \bar L(\cdot,\tau) =0$.

 We will show that  $\nabla^{g_M(-\tau)} \bar L(\cdot,\tau) \neq 0$, which will imply the splitting $M= N\times \mathbb R$ and $g_M(- \tau)=\bar h + dr^2$. For this, observe that Proposition \ref{red_dist_est} gives
\begin{eqnarray}
\bar L(x,\tau)=L(p,\tau) &\leq& 2A-\frac{1}{A} \left(1+\frac{d_{g(-\tau)}(p,q)}{\sqrt{\tau}}\right)^2, \nonumber\\
\bar L(y,\tau)=L(q,\tau) &\geq& \frac{1}{A} \left(1+\frac{d_{g(-\tau)}(p,q)}{\sqrt{\tau}}\right)^2-2A.\nonumber
\end{eqnarray}
From (\ref{assumption2}) it follows that $\bar L(x, \tau)<0<\bar L(y,\tau)$, hence $\bar L(\cdot,\tau)$ is not constant. By scaling, we may assume that $\tau=1$ and $g_M(-1)=\bar h +dr^2$.

Moreover, the restriction $f(\cdot)$ of $l_p(\cdot,1)$ on $N\times 0$ satisfies
$$\ric_{\bar h}+\hess_{\bar h} f=\frac{1}{2} \bar h,$$
hence it is a soliton function for $\bar h$. 

 It is easy to see there are $a,b\in \mathbb R$ such that $l_p((z,v),1)=f(z)+\left(\frac{|v|}{2}+a\right)^2+b$, for every $(z,v)\in N\times\mathbb R^{n-k+1}$. Hence $\nabla^{g(-1)}l_p((z,v),1)=\nabla^{\bar h}f(z) +\left(\frac{|v|}{2} + a\right)\frac{\partial}{\partial r}(v)$, where $\frac{\partial}{\partial r}$ denotes the radial vector field in $\mathbb R^{n-k+1}$. It follows that the diffeomorphisms $\phi_t$ which generate $g(t)$ are of the form $\phi_t(z,v)=(\phi_{1,t}(z),\phi_{2,t}(v))$, 
where 
\begin{eqnarray}
\frac{d}{dt} \phi_{1,t}(z)&=&-\frac{1}{t}(\nabla^{\bar h}f)(\phi_{1,t}(z)),\qquad \phi_{1,-1}(z)=z,\nonumber\\
 \frac{d}{dt} \phi_{2,t}(v)&=&-\frac{1}{t} \left(\frac{|\phi_{2,t}(v)|}{2}+a\right)\frac{\partial}{\partial r}, \qquad \phi_{2,-1}(v)=v.\nonumber
\end{eqnarray}
Therefore,  putting $h(t)=-t(\phi_{1,t})^* \bar h$ for the Ricci flow generated by $(N,\bar h, f)$ we obtain
\begin{eqnarray}
g(t)&=&-t\phi^*_t g(-1),\nonumber\\
&=&-t\phi^*_{1,t} h(-1) -t\phi^*_{2,t} dr^2,\nonumber\\
&=& h(t)+dr^2.\nonumber
\end{eqnarray}
Hence,  the flow $(M,g(t))_{t\in(-\infty,0)}$ splits for all time.
\end{proof}

\begin{theorem}\label{splitting}
 Let $(M^n,g(t))_{t\in (-\infty,0)}\in \rf$ be a non-flat gradient shrinking Ricci soliton. Then, there exists an integer $2\leq k\leq n$, a gradient shrinking Ricci soliton $(N^k, h(t))_{t\in (-\infty,0)}$ and $D=A\sqrt 2-1>0$ ($A$ is as in Proposition \ref{red_dist_est}) such that
\begin{enumerate}
 \item $(M,g(t))$ splits isometrically as $(N^k,h(t))\times (\mathbb R^{n-k}, g_{Eucl})$.
 \item There is a $V\subseteq N$ such that $S(M,g)= V\times \mathbb R^{n-k}$ and $\diam_{h(-\tau)} (V) \leq D \sqrt \tau$.
\end{enumerate}
\end{theorem}
\begin{proof}
Let $0\leq k\leq n$ be the minimal $k$ with the property that $(M^n,g(t))=(N^k,h(t))\times (\mathbb R^{n-k},g_{Eucl})$, for some non-flat gradient shrinking Ricci soliton $(N,h(t))$. Since $(M,g(t))$ is not flat, $k\geq 2$. Moreover, by the translational symmetry there exists a $V\subseteq N$ such that $S(M,g)=V\times \mathbb R^{n-k}$.

All we need to show is that $\diam_{h(-\tau)}(V)\leq D\sqrt \tau$. If this is violated for some $\tau>0$,  Lemma \ref{cone_splitting} implies that $(N^k,h(t))$ splits a line, thus contradicting the minimality of $k$.
\end{proof}

\section{The size of the singular strata.}\label{sizeof}
\subsection{Density uniqueness of tangent flows.}
Let $\mathfrak g=(M,g(t),p)_{t\in [-T,0)}\in \rf$. Given an arbitrary sequence $\tau_i\searrow 0$ consider the blow up sequence $\mathfrak g_i=(M,\tau_i^{-1}g(\tau_i t),p)$ converging to a tangent flow $\mathfrak h=(N,h(t), q) \in \rf$. As was described in the introduction, the combined work of  \cite{Naber}, \cite{EMT} and \cite{Mante_Muller}  shows that $\mathfrak h$ is in fact the Ricci flow induced by a gradient shrinking Ricci soliton $(N,h(-1), f)$ and is non-flat, provided that $p$ belongs to the singular set $\Sigma$ of $(M,g(t))_{[-T,0)}$ (see Definition \ref{singular_set}).

A tangent flow of $\mathfrak g$ may depend on the choice of the sequence $\tau_i$, and is not unique in general. However, the following theorem asserts that all tangent flows of $\mathfrak g$ should have the same asymptotic reduced volume.

\begin{theorem}\label{ARV_uniqueness}
Let $\mathfrak g=(M,g(t),p)_{t\in (-T,0)}\in \rf$ and $\mathfrak h=(N,h(t),q)_{t\in (-\infty,0)}$ be a tangent flow of $\mathfrak g$. Then
\begin{eqnarray}
\Theta_g(p)=\Theta_h(q)=\mathcal{ARV}(N,h).
\end{eqnarray}
\end{theorem}
\begin{proof}
The proof is similar to that of Lemma \ref{density_soliton}. Fix some singular reduced distance $l$ for $\mathfrak g$ and consider a sequence $\tau_i \searrow 0$ such that the blow up sequence $\mathfrak g_i=(M,\tau_i^{-1}g(\tau_i \tau),p)$ converges to $\mathfrak h=(N,h(t),q)_{t\in (-\infty,0)}$. Set $l_i(\cdot,\tau)=l(\cdot,\tau_i \tau)$ for the corresponding singular reduced distance.

By Proposition \ref{red_dist_est} we obtain that along a subsequence $l_i$ converge to some singular reduced distance $\bar l$ for $\mathfrak h$. Moreover,
\begin{eqnarray}
\rv_{\mathfrak h,\bar l}(\bar\tau)=\lim_{\tau \searrow 0} \rv_{\mathfrak g,l}(\tau),
\end{eqnarray}

By Part 3 of  Lemma \ref{density_soliton} we also obtain that $\Theta_h(q)=\rv_{\mathfrak h,\bar l}(\tau) $. Hence, since  $\Theta_g(p)=\liminf_i \Theta_{g_i}(p) \geq \Theta_h(q)$ we obtain
\begin{eqnarray}
\Theta_g(p)\geq \Theta_h(q)=\rv_{\mathfrak h,\bar l}(\tau)=\lim_{\tau\searrow 0}\rv_{\mathfrak g,l}(\tau)\geq \lim_{\tau \searrow 0 }\rv_{\mathfrak g}(\tau)= \Theta_g(p),
\end{eqnarray}
for every $\bar\tau>0$, which proves the theorem.
\end{proof}

\begin{remark}
 Compare Theorem \ref{ARV_uniqueness} with the entropy uniqueness of tangent flows observed by Mantegazza and M\"{u}ller in \cite{Mante_Muller}. The $\mathcal W$-entropy of a gradient shrinking Ricci soliton $(N,h,f)$ is defined as
\begin{eqnarray}
 \mathcal W=\int_N (R_h-|\nabla f|^2+f-n)\frac{e^{-f}}{(4\pi)^{\frac{n}{2}}} d\mu_{h},
\end{eqnarray}
where $R_h$ denotes the scalar curvature of $h$ and  $f$ is normalized so that $\int_N \frac{e^{-f}}{(4\pi)^{\frac{n}{2}}}d\mu_{h}=1$.
\end{remark}

\subsection{Estimating the size of the strata.}
 Now we are ready to prove Theorem \ref{main_theorem} and Corollary \ref{corol}. Fix a singular compact Type I Ricci flow $(M,g(t))_{t\in [-T,0)}$, and let $C>0$ be such that 
\begin{eqnarray}
\sup_M|\Rm(g(t))|_{g(t)}\leq \frac{C}{|t|},
\end{eqnarray}
for $t\in[-T,0)$. Given any $r>0$, set $g_r(t)=r^{-2}g(r^2 t)$. 

\begin{definition}
 For every $x\in \Sigma$, $r>0$ and $\delta>0$ define
\begin{eqnarray}
 S^{x,r,\delta}=\{y\in B_{g_r}(x,-1,4D),\;\; \Theta_{g_r}(y)< \Theta_g(x)+\delta \}. \nonumber
\end{eqnarray}
\end{definition}

 The sets $S^{x,r,\delta}$ are important because of the following lemma.

\begin{lemma}(Line-up lemma)\label{approx}
 For every $\epsilon,\alpha>0$ and $x\in \Sigma_j$, $j=0,...,n-2$, there is a $\delta=\delta(x,j,\epsilon, \alpha)>0$ such that for every $r\in (0,\delta)$  there exists a non-flat shrinking Ricci soliton  $(X, z(t), m)_{t\in (-\infty,0)}\in \rf$ with the following properties.
 \begin{enumerate}
 \item $(X,z(t))$ splits isometrically as $(N^{n-k}, h(t))\times (\mathbb R^k, g_{Eucl})$, for some $k\leq j$. 
 \item $m\in S(X,z)=V\times \mathbb R^k$, where $V\subseteq N$ satisfies $\diam_{h(-\tau)}V\leq D \sqrt \tau$, for all $\tau>0$. Here $D=D(n,C)$  is given by  Theorem \ref{splitting}.
 \item There is a diffeomorphism $F:B_{z}(m,-1, 5D)\rightarrow M,$ with $F(m)=x$, such that
\begin{eqnarray}
 F^{-1}(S^{x,r,\delta})\subseteq \mathcal N^{z(t)}_\epsilon(V\times \mathbb R^k),&& \textrm{and}  \nonumber\\
 (1.001)^{-2}z(t)\leq F^* g_r(t)\leq 1.001^2 z(t),&& \nonumber
\end{eqnarray}
for every $t\in[-1,-\alpha]$. Here,  $\mathcal N^{z(t)}_\epsilon(\:\cdot\:)$ denotes the $\epsilon$-neighbourhood with respect to $z(t)$.
\end{enumerate}
\end{lemma}
\begin{proof}
Fix $x\in \Sigma_j$ and $\epsilon,\alpha>0$. Arguing by contradiction and passing to a subsequence if necessary, we obtain sequences $0< r_i<\delta_i$ such that $\delta_i,r_i\searrow 0$ such that:
\begin{enumerate}
\item[(i)] There is a non-flat shrinking Ricci soliton $(X, z(t), m)_{t\in (-\infty,0)}\in \rf$ satisfying (1), (2) and  $(M,g_{r_i}(t),x)\rightarrow (X, z(t), m)$. Moreover, there are diffeomorphisms $F_i:B_{z}(m,-1, 5D)\rightarrow M,$ with $F_i(m)=x$ and 
\begin{eqnarray}
(1.001)^{-2}z(t)\leq F_i^* g_{r_i}(t)\leq 1.001^2z(t),\nonumber
\end{eqnarray}
for every $t\in [-1,-\alpha]$.
\item[(ii)] There are  sequences $t_i\in [-1,\alpha]$  and $y_i\in B_{g_{r_i}}(x,-1,4D)$ such that $t_i\rightarrow \bar t$, $F_i^{-1}(y_i)\rightarrow y_\infty$ and $\Theta_{g_{r_i}}(y_i)<\Theta_g(x)+\delta_i$, but $F^{-1}(y_i) \notin \mathcal N^{z(t_i)}_\epsilon (V\times \mathbb R^k )$.
\end{enumerate}
It follows that $y_\infty \notin \mathcal N_\epsilon^{z(\bar t)}(S(X,z))$. However, the lower semicontinuity of the density under Cheeger-Gromov convergence and Theorem \ref{ARV_uniqueness} imply that 
\begin{eqnarray}
\Theta_z(y_\infty)\leq \liminf_i \Theta_{g_{r_i}} (y_i)\leq\Theta_g(x)=\Theta_z(m).
\end{eqnarray}
Hence $y_\infty \in S(X,z)$, which is a contradiction.
\end{proof}

\begin{lemma}(Covering lemma)\label{covering_lemma}
Let $(X,z(t),m)_{t\in (-\infty,0)}\in \rf$ be a non-flat, shrinking Ricci soliton satisfying properties (1) and (2) of Lemma \ref{approx} and $s=j+\varepsilon$, $j=0,...,n-2$, $\varepsilon>0$. Then, there is $\sigma_s\in (0,\frac{1}{2})$ such that for every $\sigma\in (0,\sigma_s]$, $\rho \in (0,4D)$, $\alpha\in (0,\left(\frac{\sigma\rho}{2D}\right)^2 )$ and $x\in \mathbb R^k$ we obtain the covering
$$\mathcal N^{z(-\alpha)}_{\frac{\sigma\rho}{4}}(V\times B_\rho(x))\subseteq\bigcup_{l=1}^{P(\sigma)} B_{z}((q,x_l),-\alpha,\frac{\sigma \rho}{1.001^2}),$$
for some $x_l\in B_\rho(x)\subseteq \mathbb R^k$, $l=1,\ldots, P(\sigma)$, and  $q\in V$. Moreover,  $P(\sigma)$ satisfies
\begin{eqnarray}
P(\sigma) \sigma^j &\leq& C(n) ,\label{cov1}\\
P(\sigma)\sigma^s &\leq& \frac{1}{2}. \label{cov2} 
\end{eqnarray}

\end{lemma}
\begin{proof}
There is a $C(n)>0$ such that,  for every $\sigma>0$ and  $k\leq j$, we can cover the unit ball $\overline{B_1(0)}\subseteq \mathbb R^k$ with $P(\sigma)$ balls of radius $\frac{\sigma}{4}$ such that
\begin{eqnarray}
 P(\sigma) \sigma^j &\leq& C(n) .
 \end{eqnarray}
 Moreover, we can chose $\sigma >0$ small enough in order to satisfy 
 \begin{eqnarray}
 P(\sigma)\sigma^s &\leq& \frac{1}{2}. 
\end{eqnarray}
Hence,  we can cover any $B_\rho (x)\subseteq \mathbb R^k$ by $P(\sigma)$ balls of radius $\frac{\sigma\rho}{4}$, 
\begin{eqnarray}
 B_\rho(x) \subseteq \bigcup_{l=1}^{P(\sigma)} B_{\frac{\sigma\rho}{4}}(x_l), \label{covering}
\end{eqnarray}
for some $x_l\in B_\rho(x)$, $l=1,\ldots,P(\sigma)$, so that (\ref{cov1}) and (\ref{cov2}) hold.

Now, suppose that $m=(q,0)\in V\times \mathbb R^k$, let $y=(p,\bar x) \in \mathcal N^{z(-\alpha)}_{\frac{\sigma\rho}{4}}(V\times B_\rho (x))$ and choose $y'=(q',x')\in V\times B_\rho (x)$ such that $d_{z(-\alpha)}(y,y')<\frac{\sigma \rho}{4}$. Then, from (\ref{covering}) there is $x_{l_0}$ such that $x'\in B_{\frac{\sigma\rho}{4}}(x_{l_0})$.

We then compute, using Theorem \ref{splitting},
\begin{eqnarray}
 d_{z(t)}(y,(q,x_{l_0}))&\leq& d_{z(t)}(y,y') + \sqrt{  (\diam_{h(t)}V)^2 + \left(\frac{\sigma\rho}{4}\right)^2 },\\
&\leq& d_{z(t)}(y,y') + \sqrt{  -t D^2 + \left(\frac{\sigma\rho}{4}\right)^2 }.
\end{eqnarray}

Putting  $-t=\alpha\leq \left(\frac{\sigma\rho}{2D}\right)^2$ we obtain
\begin{eqnarray}
 d_{z(-\alpha)}(y,(q,x_{l_0}))\leq \frac{\sigma\rho}{4}(1+\sqrt 5)<\frac{\sigma\rho}{1.001^2},
\end{eqnarray}
which suffices to prove the lemma. 
\end{proof}

\begin{proof}[Proof of Theorem  \ref{main_theorem}]
Fix $j$ and $s=j+\varepsilon$, as in the statement of the theorem. Let $\sigma:=\sigma_s\in (0,\frac{1}{2})$ and $D:=D(n,C)>0$ be the constants defined in Theorems \ref{covering_lemma} and \ref{splitting} respectively, and set $\alpha:=\sigma^2$, $\epsilon_s=\frac{2D 1.001\sigma}{4}$.
 
  Define, for each $i,m\geq 1$,
\begin{eqnarray}
D_{i}=\left\{ x\in \Sigma_j, \delta(x,j,\epsilon_s,\alpha)\geq \frac{1}{i}\right\}
\end{eqnarray}
and
\begin{eqnarray}
 S^{i,m}=\left\{x \in D_i, \;\; \Theta(x)\in \left[\frac{m-1}{i}, \frac{m}{i}\right)\right\}.
\end{eqnarray}
Since $\Theta_g(\;\cdot\;)\in [0,1]$, by Remark \ref{density_values}, it follows that $D_i = \bigcup_{m=1}^i S^{i,m}$.

In the following, we fix $i\geq 1$.  We will show that there exists $L_i>0$  such that for every $q\geq 0$ and $\bar\tau \in [ \alpha i^{-2},i^{-2}]$ 
\begin{eqnarray}
D_i \subseteq \bigcup_{l=1}^{Q_q} \overline{B_g (x_{q,l}, -\alpha^q \bar\tau, 2D\sqrt{\alpha^q \bar\tau})}, \label{kalyma}
\end{eqnarray}
and $Q_q (2D \sqrt{\alpha^q \bar\tau})^s \leq 2^{-q} L_i$. 

Assuming for now this is true, for each $\tau \in (0,1]$ define the closed sets
\begin{eqnarray}
C_i (\tau)&=&\left\{
\begin{array}{ll}
\bigcup_{l=1}^{Q_q} \overline{B_g (x_l, -\alpha^q \bar\tau, 2D\sqrt{\alpha^q \bar\tau})}, & \tau=\alpha^q \bar \tau , \;\bar\tau \in [\alpha i^{-2} , i^{-2}],\; q\geq 0,\\
& \\
C_i(i^{-2}),  & \tau \geq i^{-2},
\end{array}
\right.\nonumber\\
C_i &=&\bigcap_{\tau>0} C_i(\tau), \nonumber
\end{eqnarray}

Using the property of the cover (\ref{kalyma}) and the non-inflating property of the Ricci flow we compute
\begin{eqnarray}
\vol_{g(-\tau)} (C_i) &\leq& \vol_{g(-\tau)} (C_i (\tau)) \\
&=& \vol_{g(-\alpha^q \bar\tau)} (C_i (\alpha^q \bar \tau) )\\
&\leq& Q_q \kappa_0 (2D \sqrt{\alpha^q \bar \tau})^{n-s} (2D \sqrt{\alpha^q \bar \tau})^s \label{vol_cover}\\
&\leq& L_i\kappa_0 2^{-q} (2D\sqrt{\alpha^q\bar\tau})^{n-s}.
\end{eqnarray}

Since we can uniquely write $\tau=\alpha^q \bar\tau$, for some $\bar\tau\in [\alpha i^{-2},i^{-2}]$, it follows that
\begin{eqnarray}
 \frac{\vol_{g(-\tau)} (C_i) }{\tau^{\frac{n-s}{2}}} \leq L_i \kappa_0 \left( \frac{\sigma}{i}\right)^{2\beta} \tau^{\beta},
\end{eqnarray}
where $\beta=-\frac{1}{2\log_2\sigma}$. The theorem is then proven setting $A_i=C_i \cap \bar \Sigma_j$. Observe also that $\Sigma_j\subset \cup_{i=1}^\infty A_i$, since $\Sigma_j = \cup_{i=1}^\infty D_i$. 

Moreover, since along Ricci flow the estimate $R \geq -\frac{n}{2(t+T)}$ is valid for all $t\in [-T,0)$, it follows that there is a $C=C(n,T)>0$ such that $\vol_{g(t)}(\bar\Sigma_j \setminus A_i)\leq C \vol_{g(-T)}(\Sigma_j\setminus A_i)$ for every $i$. Choosing $i$ large enough so that $C \vol_{g(-T)}(\bar\Sigma_j\setminus A_i)<\delta$ suffices to prove (\ref{small_vol}).

Now we will prove (\ref{kalyma}) by induction in $q$. First, we claim that there exist $L^m_i,Q^m_0>0$ and $p_{m,1},\ldots,p_{m,Q^m_0}\in M$ such that for every $\bar\tau\in [\frac{\alpha}{i^2},\frac{1}{i^2}]$
\begin{eqnarray}
 S^{i,m}\subseteq\bigcup_{l=1}^{Q^m_0} B_g(p_{m,l},-\bar\tau,2D\sqrt{\bar\tau}), \label{tricky_cover}
\end{eqnarray}
and $Q^m_0(2D\sqrt{\bar\tau})^s \leq L^m_i$.  To see this, observe that there exists $R_i>0$ (depending also on the the Type I curvature bound $C$) such that for every $p\in M$ and $\bar\tau\in [\frac{\alpha}{i^2},\frac{1}{i^2}]$
\begin{eqnarray}
 B_g(p,-\frac{1}{i^2},R_i)\subseteq B_g(p, -\bar\tau,2 D\sqrt{\bar\tau}),\label{cover}
\end{eqnarray}
Hence, choosing a cover of $S^{i,m}$ by $Q^m_0$ balls $B_g(x_l,-\frac{1}{i^2},R_i)$, $l=1,\ldots, Q^m_0$, we immediately obtain (\ref{tricky_cover}), for $L^m_i=Q^m_0(\frac{2D}{i})^s $.

Now, let $q\geq 1$, $\tau\in [\frac{\alpha^{q+1}}{i^2},\frac{\alpha^q}{i^2}]$, and suppose there exist $p_l\in M$, $l=1,\ldots,Q_q^m$ such that 
\begin{eqnarray}
 S^{i,m}\subseteq \bigcup_{l=1}^{Q_q^m} B(p_l,-\tau, 2D\sqrt{\tau}),
\end{eqnarray}
$B(p_l,-\tau, 2D\sqrt{\tau})\cap S^{i,m}\neq \emptyset$ for every $l$ and $Q_q^m(2D\sqrt \tau)^s \leq 2^{-q}L^m_i$.

Choose any such ball $B_g(p_{l_0}, -\tau, 2D\sqrt{\tau})$ and $x\in S^{i,m}\cap B_g(p_{l_0},-\tau,2D\sqrt{\tau}) $. Then, from the definitions of $S^{i,m}$ and $S^{x,\sqrt \tau, i^{-1}}$ it follows that
\begin{eqnarray}
S^{i,m}\cap B_{g_{\sqrt{\tau}}}(x,-1,4D)\subseteq S^{x,\sqrt{\tau},i^{-1}}.
\end{eqnarray}

Hence, there exist $k\leq j$, $X=N^{n-k}\times \mathbb R^k$,  $z(t)=h(t)+g_{Eucl}$, $V\subset N$ and $F$ as in Lemma \ref{approx} such that
$$F^{-1}(S^{i,m}\cap B_{g_{\sqrt{\tau}}}(p_{l_0},-1,2D))\subseteq F^{-1}(S^{i,m}\cap B_{g_{\sqrt{\tau}}}(x,-1,4D))\subseteq  \mathcal N^{z(-\alpha)}_{\epsilon_s}(V\times \mathbb R^k).$$

Moreover, note that there is a ball $B_{ 2D(1.001)}\subseteq \mathbb R^j$ such that
\begin{eqnarray}
 F^{-1}(S^{i,m}\cap B_{g_{\sqrt{\tau}}}(p_{l_0},-1,2D))\subseteq  \mathcal N^{z(-\alpha)}_{\epsilon_s}(V\times B_{2D(1.001)}).
\end{eqnarray}
Hence, putting $\rho=2D(1.001)$ in the covering Lemma \ref{covering_lemma} we obtain
\begin{eqnarray}
 F^{-1}(S^{i,m}\cap B_{g_{\sqrt{\tau}}}(p_{l_0},-1,2D))\subseteq \bigcup_{a=1}^{P(\sigma)} B_z((q,y_a), -\alpha, \frac{2\sigma D}{1.001}), 
\end{eqnarray}
for $y_a\in \mathbb R^k$ and $P(\sigma)\sigma^s\leq \frac{1}{2}$, since $\alpha=\sigma^2 < (\frac{\sigma \rho}{2D})^2=\sigma^2 1.001^2$.

Thus, there exist $o_l\in M$ such that 
\begin{eqnarray}
 S^{i,m}\cap B_{g_{\sqrt{\tau}}}(p_{l_0},-1,2D)\subseteq \bigcup_{l=1}^{P(\sigma)} B_{g_{\sqrt{\tau}}}(o_l, -\sigma^2, 2\sigma D),
\end{eqnarray} 
$o_l\in M$, which implies that there exist $p'_l \in M$, $l=1,\ldots,Q_{q+1}^m$, $Q_{q+1}^m\leq Q_q^mP(\sigma)$ such that
\begin{eqnarray}
 S^{i,m}\subseteq \bigcup_{l=1}^{Q_{q+1}^m} B(p'_l,- \alpha \tau, 2D\sqrt{\alpha\tau}),
\end{eqnarray}
$B(p'_l,-\alpha\tau, 2D\sqrt{\alpha\tau})\cap S^{i,m}\neq \emptyset$ for every $l$ and $Q_{q+1}^m (2D\sqrt{\alpha\tau})^s \leq 2^{-(q+1)}L_i$. This proves (\ref{kalyma}) for $L_i=\sum_m L_i^m$ and $Q_q =\sum_m Q_q^m$.

In order to prove (\ref{size_est_3}), take any $\epsilon,\alpha>0$, $x \in \Sigma_0$ and set $\bar\delta=\delta(x,0,\epsilon,\alpha)$. From Theorem \ref{splitting} and Lemma \ref{approx},  for every $\tau\in(0,\bar\delta^2]$ there is a non-flat shrinking Ricci soliton $(X,z(t), m)_{t\in (-\infty,0)}$ such that $\diam_{z(t)}S(X,z) \leq D\sqrt{-t}$, and a diffeomorphism $F:B_z(m,-1,5D)\rightarrow M$ with $F(m)=x$ satisfying
\begin{eqnarray}
F^{-1}(S^{x,\sqrt \tau, i^{-1}}) \subset \mathcal N_{\epsilon}^{z(t)}( S(X,z) ),\label{ena}\\
(1.001)^{-2}z(t) \leq F^* g_{\sqrt \tau}(t) \leq 1.001^2 z(t),\label{dyo}
\end{eqnarray}
for every $t\in [-1,-\alpha]$.

Now, take any $\tau\in (0,\bar\delta^2]$ and let $y'\in B_g(x,-\tau,4D \sqrt \tau)\cap \{ y\in M, \; \Theta_g(y)\leq \Theta_g(x)\}$. 
Since $ \{ y\in  B_g(x,-\tau,4D \sqrt \tau), \; \Theta_g(y)\leq \Theta_g(x)\} \subset S^{x,\sqrt \tau,\bar\delta}$ it follows from (\ref{ena}) that $F^{-1}(y')\in B_z(m,-\lambda, \sqrt{\lambda}D+\epsilon)$, for every $\lambda\in [\alpha,1]$.

  Then, by (\ref{dyo}), $y'\in B_g (x,  -\lambda \tau, 1.001(\sqrt{\lambda}D+\epsilon) \sqrt{\tau})$. Moreover, there is a $\bar\lambda \in [\alpha , 1)$ (independent of $\tau$) such that for every $\lambda \in[\bar \lambda, 1]$, 
 \begin{eqnarray}
 B_g (x,  -  \lambda\tau, 1.001(\sqrt{\lambda}D+\epsilon) \sqrt{\tau})\subset B_g (x,  -\lambda \tau, 4D \sqrt{\lambda\tau}).
 \end{eqnarray}
  
 We conclude that for every $\tau\in (0,\bar\delta^2]$ and $\lambda\in [\bar\lambda,1]$
 \begin{eqnarray}
  B_g(x,-\tau,4D \sqrt \tau)\cap \{ y\in M, \; \Theta_g(y)\leq \Theta_g(x)\} \subset  B_g (x,  - \lambda\tau, 4D \sqrt{\lambda\tau}),
 \end{eqnarray}
 which suffices to prove (\ref{size_est_3}) for $\bar\tau=\bar\delta^2$ and $R_0=4D$.
\end{proof}

Below, we finish by proving Corollary \ref{corol}.

\begin{proof}[Proof Corollary \ref{corol}]
First of all, notice that $\Sigma$ is closed, hence $\overline{\Sigma_j}=\Sigma$. Now, estimate (\ref{cor_est}) follows from Theorem \ref{main_theorem} by setting $s=j+2\varepsilon$.

In general $\Sigma=\Sigma_{n-2}$, since every shrinking Ricci soliton which splits more than $n-2$ Euclidean factor is flat. Therefore, the volume estimate of Case 1 follows by setting $j=n-2$.

In Case 2, every tangent flow should have vanishing Weyl curvature. Thus, it should be isometric either to flat $\mathbb R^n$ (the Gaussian soliton), or to quotiens of $S^{n-1}\times \mathbb R$ or $S^n$, by \cite{Weylflat}. Hence $\Sigma=\Sigma_1$ and the volume estimate follows by setting $s=1+2\varepsilon$.

\end{proof}

\providecommand{\bysame}{\leavevmode\hbox to3em{\hrulefill}\thinspace}
\providecommand{\MR}{\relax\ifhmode\unskip\space\fi MR }
\providecommand{\MRhref}[2]{%
  \href{http://www.ams.org/mathscinet-getitem?mr=#1}{#2}
}
\providecommand{\href}[2]{#2}


\begin{thebibliography}{10}

\bibitem{Almgren}
F.~J. Almgren, Jr., \emph{{$Q$} valued functions minimizing {D}irichlet's
  integral and the regularity of area minimizing rectifiable currents up to
  codimension two}, Bull. Amer. Math. Soc. (N.S.) \textbf{8} (1983), no.~2,
  327--328. \MR{684900 (84b:49052)}

\bibitem{CheegerColding}
Jeff Cheeger and Tobias~H. Colding, \emph{On the structure of spaces with
  {R}icci curvature bounded below. {I}}, J. Differential Geom. \textbf{46}
  (1997), no.~3, 406--480. \MR{1484888 (98k:53044)}

\bibitem{ChHasNab1}
Jeff Cheeger, Robert Haslhofer, and Aaron Naber, \emph{Quantitative
  stratification and the regularity of mean curvature flow}, Geom. Funct. Anal.
  \textbf{23} (2013), no.~3, 828--847. \MR{3061773}

\bibitem{ChHasNab2}
\bysame, \emph{Quantitative stratification and the regularity of harmonic map
  flow}, Calc. Var. Partial Differential Equations \textbf{53} (2015), no.~1-2,
  365--381. \MR{3336324}

\bibitem{ChNab1}
Jeff Cheeger and Aaron Naber, \emph{Lower bounds on {R}icci curvature and
  quantitative behavior of singular sets}, Invent. Math. \textbf{191} (2013),
  no.~2, 321--339. \MR{3010378}

\bibitem{ChNab2}
\bysame, \emph{Quantitative stratification and the regularity of harmonic maps
  and minimal currents}, Comm. Pure Appl. Math. \textbf{66} (2013), no.~6,
  965--990. \MR{3043387}

\bibitem{ChNabVAl}
Jeff Cheeger, Aaron Naber, and Daniele Valtorta, \emph{Critical sets of
  elliptic equations}, Comm. Pure Appl. Math. \textbf{68} (2015), no.~2,
  173--209. \MR{3298662}

\bibitem{WangChen}
Xiuxiong Chen and Bing Wang, \emph{On the conditions to extend {R}icci
  flow({III})}, Int. Math. Res. Not. IMRN (2013), no.~10, 2349--2367.
  \MR{3061942}

\bibitem{ricciflow}
Bennett Chow, Sun-Chin Chu, David Glickenstein, Christine Guenther, James
  Isenberg, Tom Ivey, Dan Knopf, Peng Lu, Feng Luo, and Lei Ni, \emph{The
  {R}icci flow: techniques and applications. {P}art {I}}, Mathematical Surveys
  and Monographs, vol. 135, American Mathematical Society, Providence, RI,
  2007, Geometric aspects. \MR{2302600 (2008f:53088)}

\bibitem{Collins}
Tristan~C. Collins and Valentino Tosatti, \emph{K{\"a}hler currents and null
  loci}, Inventiones mathematicae \textbf{202} (2015), no.~3, 1167--1198.

\bibitem{EMT}
Joerg Enders, Reto M{\"u}ller, and Peter~M. Topping, \emph{On type-{I}
  singularities in {R}icci flow}, Comm. Anal. Geom. \textbf{19} (2011), no.~5,
  905--922. \MR{2886712}

\bibitem{Federer}
Herbert Federer, \emph{The singular sets of area minimizing rectifiable
  currents with codimension one and of area minimizing flat chains modulo two
  with arbitrary codimension}, Bull. Amer. Math. Soc. \textbf{76} (1970),
  767--771. \MR{0260981 (41 \#5601)}

\bibitem{FIK}
Mikhail Feldman, Tom Ilmanen, and Dan Knopf, \emph{Rotationally symmetric
  shrinking and expanding gradient kähler-ricci solitons}, J. Differential
  Geom. \textbf{65} (2003), no.~2, 169--209.

\bibitem{HN}
R.~{Haslhofer} and A.~{Naber}, \emph{{Weak solutions for the Ricci flow I}},
  ArXiv e-prints (2015).

\bibitem{HM}
Robert Haslhofer and Reto M{\"u}ller, \emph{A compactness theorem for complete
  {R}icci shrinkers}, Geom. Funct. Anal. \textbf{21} (2011), no.~5, 1091--1116.
  \MR{2846384 (2012k:53065)}

\bibitem{KL1}
B.~{Kleiner} and J.~{Lott}, \emph{{Singular Ricci flows I}}, ArXiv e-prints
  (2014).

\bibitem{Mante_Muller}
Carlo Mantegazza and Reto M{\"u}ller, \emph{Perelman's entropy functional at
  {T}ype {I} singularities of the {R}icci flow}, J. Reine Angew. Math.
  \textbf{703} (2015), 173--199. \MR{3353546}

\bibitem{Naber}
Aaron Naber, \emph{Noncompact shrinking four solitons with nonnegative
  curvature}, J. Reine Angew. Math. \textbf{645} (2010), 125--153. \MR{2673425
  (2012k:53081)}

\bibitem{Perelman1}
G.~{Perelman}, \emph{{The entropy formula for the Ricci flow and its geometric
  applications}}, ArXiv Mathematics e-prints (2002).

\bibitem{SU}
Richard Schoen and Karen Uhlenbeck, \emph{A regularity theory for harmonic
  maps}, J. Differential Geom. \textbf{17} (1982), no.~2, 307--335. \MR{664498
  (84b:58037a)}

\bibitem{SimonL}
Leon Simon, \emph{Theorems on the regularity and singularity of minimal
  surfaces and harmonic maps}, Geometry and global analysis ({S}endai, 1993),
  Tohoku Univ., Sendai, 1993, pp.~111--145. \MR{1361175 (96j:58045)}

\bibitem{White}
Brian White, \emph{Stratification of minimal surfaces, mean curvature flows,
  and harmonic maps}, J. Reine Angew. Math. \textbf{488} (1997), 1--35.
  \MR{1465365 (99b:49038)}

\bibitem{Yokota}
Takumi Yokota, \emph{Perelman's reduced volume and a gap theorem for the
  {R}icci flow}, Comm. Anal. Geom. \textbf{17} (2009), no.~2, 227--263.
  \MR{2520908 (2010g:53125)}

\bibitem{Zhang}
Qi~S. Zhang, \emph{Bounds on volume growth of geodesic balls under {R}icci
  flow}, Math. Res. Lett. \textbf{19} (2012), no.~1, 245--253. \MR{2923189}

\bibitem{Weylflat}
Zhu-Hong Zhang, \emph{Gradient shrinking solitons with vanishing {W}eyl
  tensor}, Pacific J. Math. \textbf{242} (2009), no.~1, 189--200. \MR{2525510
  (2010f:53116)}

\end{thebibliography}
\end{document}